\documentclass{amsart}
\usepackage[utf8]{inputenc}
\usepackage{colonequals}
\usepackage{amssymb,bm,amsfonts, stmaryrd, amsmath, amsthm, enumerate, mathtools,comment}
\usepackage[all]{xy}
\makeatletter
\@namedef{subjclassname@2020}{\textup{2020} Mathematics Subject Classification}
\makeatother
\usepackage[colorlinks]{hyperref}
\hypersetup{
 colorlinks=true,
 linkcolor=blue,
 filecolor=magenta, 
 urlcolor=cyan,
}
\usepackage{wrapfig,tikz, tikz-cd}
\usetikzlibrary{arrows}
\usepackage{mathrsfs}
\usepackage[capitalise]{cleveref}
\crefformat{equation}{(#2#1#3)}
\crefrangeformat{equation}{(#3#1#4--#5#2#6)}
\crefformat{enumi}{(#2#1#3)}
\crefrangeformat{enumi}{(#3#1#4--#5#2#6)}
\newtheorem{theorem}{Theorem}[section]
\newtheorem{lemma}[theorem]{Lemma}
\newtheorem{proposition}[theorem]{Proposition}
\newtheorem{corollary}[theorem]{Corollary}
\newcounter{intro}

\newtheorem{introthm}[intro]{Theorem}
\newtheorem{introcor}[intro]{Corollary}

\theoremstyle{definition}
\newtheorem{definition}[theorem]{Definition}
\newtheorem{example}[theorem]{Example}
\newtheorem{remark}[theorem]{Remark}

\newtheorem{chunk}[theorem]{}

\newtheorem{thm}{Theorem}[subsection]

\theoremstyle{definition}

\newtheorem{ch}[thm]{}

\title[Thick subcategories over a Koszul complex]{Classifying thick subcategories over a Koszul complex via the curved BGG correspondence}

\author[Jian Liu]{Jian Liu}
\address{School of Mathematics and Statistics, and Hubei Key Laboratory of Mathematical Sciences, Central China Normal University,  Wuhan 430079, P.R. China}
\email{jianliu@ccnu.edu.cn}

\author[Josh Pollitz]{Josh Pollitz}
\address{
Mathematics Department, 
Syracuse University, 
Syracuse, NY 13244 U.S.A.}
\email{jhpollit@syr.edu}


\DeclareMathOperator{\cone}{cone}
\DeclareMathOperator{\rank}{rank}

\DeclareMathOperator{\embdim}{embdim}

\DeclareMathOperator{\h}{H}

\newcommand{\Z}{\mathbb{Z}}

\newcommand{\T}{\mathsf{T}}
\newcommand{\D}{\mathsf{D}}
\DeclareMathOperator{\Df}{\mathsf{{D}^{f}}}

\pgfdeclarelayer{bg}    
\pgfsetlayers{bg,main}

\newcommand{\del}{\partial}
\newcommand{\f}{{\bm{f}}}

\newcommand{\p}{\mathfrak{p}}

\DeclareMathOperator{\Ima}{Im}

\DeclareMathOperator{\id}{id}
\DeclareMathOperator{\Sing}{Sing}

\DeclareMathOperator{\cx}{cx}

\DeclareMathOperator{\spec}{Spec}
\DeclareMathOperator{\Proj}{Proj}

\newcommand{\Cat}{\mathsf{C}}

\DeclareMathOperator{\Hom}{Hom}
\DeclareMathOperator{\Ext}{Ext}

\DeclareMathOperator{\V}{V}

\DeclareMathOperator{\thick}{\mathsf{thick}}

\DeclareMathOperator{\supp}{supp}

\newcommand{\shift}{{\mathsf{\Sigma}}}
\DeclareMathOperator{\RHom}{\mathsf{RHom}}

\newcommand{\xra}{\xrightarrow}

\newcommand{\perf}{\mathsf{perf}}
\newcommand{\proj}{\mathsf{proj}}

\newcommand{\hc}{\operatorname{HH}}
\newcommand{\env}[2][]{{#2}_{#1}^{\operatorname{e}}}
\newcommand{\lotimes}{\otimes^{\operatorname{L}}}

\begin{document}

\begin{abstract}
In this work we classify the thick subcategories of the bounded derived category of dg modules over a Koszul complex on any list of elements in a regular ring. This simultaneously recovers a theorem of Stevenson when the list of elements is a regular sequence and the classification of thick subcategories for an exterior algebra over a field (via the BGG correspondence). One of the major ingredients is a classification of thick tensor submodules of perfect curved dg modules over a graded commutative noetherian ring concentrated in even degrees, recovering a theorem of Hopkins and Neeman. We give several consequences of the classification result over a Koszul complex, one being that the lattice of thick subcategories of the bounded derived category is fixed by Grothendieck duality. 
\end{abstract}

\maketitle

%%%%%%%%%%%%%%%%%%%%%%%%%%%%%%%%%%%%%%%%%%%%%%%%%%%%%%
\section*{Introduction}
%%%%%%%%%%%%%%%%%%%%%%%%%%%%%%%%%%%%%%%%%%%%%%%%%%%%%%

The problem of classifying thick (or localizing) subcategories of a given triangulated category has maintained a great deal of interest---in a variety of different areas---since the work of Devinatz, Hopkins, and Smith \cite{DHS} in stable homotopy theory; see, for instance, \cite{BIKann, BIKtop, CI, Hopkins, Neeman92, Stevenson:2014b, Stevenson:2014a,  Takahashi, Thomason}. Perhaps the most relevant to the present article is a theorem of Hopkins \cite{Hopkins} and Neeman \cite{Neeman92} that establishes a one-to-one correspondence between the thick subcategories of the 
 homotopy category of perfect complexes over a commutative noetherian ring and the specialization closed subsets of the prime spectrum of the ring. Thomason \cite{Thomason} later extended this result to quasi-compact and quasi-separated schemes, provided the thick subcategories are closed under a certain tensor action.  These foundational results have been widely applied in both algebra and topology.

 Let $R$ be a commutative noetherian ring, and write $\Df(R)$ for its bounded derived category; the objects are $R$-complexes with finitely generated total homology and morphisms are obtained by formally inverting quasi-isomorphisms. The theorem of Hopkins--Neeman yields a classification of the thick subcategories of $\Df(R)$ when $R$ is a regular ring. In the singular setting, the most notable classification of the thick subcategories of $\Df(R)$ is due to Stevenson~\cite{Stevenson:2014b} which establishes a classification when $R$ is a complete intersection in the sense that it is a regular ring modulo a regular sequence; see also \cite{Briggs-Walker}. In this article we are concerned with a more general setting that recovers Stevenson's result as well as the classification of thick subcategories over an exterior algebra over a field simultaneously. 
 
 Let $E$ be a Koszul complex on a list of elements $f_1,\ldots,f_n$ regarded as a dg $R$-algebra in the usual way; cf.\@ \cref{Section-thick subcategories}. Let $\Df(E)$ denote the derived category of dg $E$-modules whose homology is finitely generated over $R$. The main result of the article classifies the thick subcategories of $\Df(E),$ when $R$ is a regular ring; specializing to when $\f$ is an $R$-regular sequence recovers Stevenson's theorem, and specializing to $\f=0$ one obtains the classification of thick subcategories for an exterior algebra over a regular ring.  
 
 The classification is in terms of the support of the pair $(A,w)$, denoted $\supp(A,w)$, where $A=R[\chi_1,\ldots,\chi_n]$ with  $w=f_1\chi_1+\cdots +f_n\chi_n$ and where each $\chi_i$ has homological degree $-2$; see \cref{defn_sing} (as well as \cref{sing_locus}) for a precise definition of $\supp(A,w)$, and see \cref{decomposition} for a set theoretic description of it in this setting.  Finally, let $\V_E(M)$ denote the cohomological support of $M$ over $E$ as defined in \cite{Pol,Po} (inspired by the ones defined by Avramov and Buchweitz~\cite{Av2,ABInvent} over a complete intersection ring); cf.\@ \cref{Hochschild}. 

 \begin{introthm}\label{main_intro_thm}
     There is an inclusion preserving bijection  
      \[
\begin{tikzcd}
\biggm\{\begin{matrix}\text{Thick subcategories}\\ \text{of } \Df(E) \end{matrix} \biggm\}\arrow[r,shift left=0.8ex,"\sigma"]&\biggm\{\begin{matrix}\text{Specialization closed}\\ \text{subsets of } \supp(A,w)\end{matrix} \biggm\}\arrow[l,shift left=0.8ex,"\tau"]
\end{tikzcd}\,,
\]
where the inverse bijections are given by
\[
\sigma(\T)=\bigcup_{M\in \T}\V_E(M) \quad\text{and}\quad \tau(\mathcal{V})=\{M\in \Df(E)\mid \V_E(M)\subseteq \mathcal{V}\}\,.
\]
 \end{introthm}

Specializing the theorem to when $\f$ is an $R$-regular sequence recovers the classification of Stevenson in \cite{Stevenson:2014b} for a complete intersection ring in \cref{cor_ci}. In the case $\f=0$, we have that $E$ is an exterior algebra over a regular ring and \cref{main_intro_thm} establishes a bijection between thick subcategories of $\Df(E)$ and the specialization closed subsets of $\spec(A);$ see \cref{cor_zero}. 

There are a number of consequences of \cref{main_intro_thm} on the homological algebra over $E$. For example, whenever $\supp(A,w)$ is closed there is a \emph{classical generator} for $\Df(E)$; cf.\@ \cref{cor_generator}. This applies when $\f=0$, or when $\f\neq 0$ and, for example, $R$ is further assumed to be quasi-excellent and local. 
In a different direction, recall that $E$ is a Gorenstein dg algebra in the sense of \cite{FJ}, and hence, $(-)^\vee=\RHom_E(-,E)$ is a self (quasi-)inverse auto-equivalence of $\Df(E)$. 
\begin{introcor}\label{intro_cor}
   For each thick subcategory $\T$ of $\Df(E)$, there is an equality $\T=\T^\vee$. 
\end{introcor}

For complete intersection rings this was established in \cite{Stevenson:2014b} via Stevenson's classification, and later also shown in \cite{LP} using a local-to-global principle from \cite{BIK15}. This also shows symmetric complexity holds over $E$; recall the complexity of a pair $M,N$ in $\Df(E)$ measures the polynomial rate of growth of the minimal number of generators of $\Ext_E^n(M,N)$ as $n\to \infty$. A consequence of \cref{intro_cor} (given in \cref{cor_sym_cx}) is that $\cx_E(M,N)=\cx_E(N,M)$, surprising (of course) because of the highly non-symmetric behavior of $\Ext(-,-).$
 The results above build on the philosophy from \cite{ABJA,Gu,Po,Tate} that the nice homological properties of a complete intersection ring are merely inherited from the well-behaved homological properties of the Koszul complex over a regular ring.

The proof of \cref{main_intro_thm}, presented in \cref{Section-thick subcategories}, depends on two major ingredients. The first is a \emph{curved} version of the BGG correspondence from Martin's thesis~\cite{Martin}. Namely, this establishes an equivalence of triangulated categories between $\Df(E)$ and the category of perfect curved dg $A$-modules having curvature $w$, denoted $\perf(A,w)$; see \cref{curved_dg}. The second ingredient is a classification of thick subcategories of $\perf(A,w)$, similar to the ones of Hopkins--Neeman--Thomason, which is given in \cref{classification_reg}. This classification follows from a more general classification result as described below. 

Let $A$ be a graded algebra that is commutative noetherian and concentrated in even degrees. Fix $w\in A_{-2}$, and we consider $\perf(A,w)$ as a tensor module over the tensor triangulated category $\perf(A,0)$ in the sense of \cite{Stevenson:2013a}; cf.\@ \cref{tensor_module}. 

\begin{introthm}
    \label{intro_thm_2}
In the notation above,  there are inclusion preserving bijections 
\[
\begin{tikzcd}
\biggm\{\begin{matrix}\text{Thick tensor submodules}\\ \text{of } \perf(A,w) \end{matrix} \biggm\}\arrow[r,shift left=0.8ex,"\sigma"]&\biggm\{\begin{matrix}\text{Specialization closed subsets}\\ \text{of } \supp(A,w)\end{matrix} \biggm\}\arrow[l,shift left=0.8ex,"\tau"]
\end{tikzcd}\,,
\]
where the inverse bijections are given by
\[
\sigma(\T)=\bigcup_{P\in \T}\supp_{(A,w)}(P)\quad\text{and}\quad \tau(\mathcal{V})=\{P\in \perf(A,w)\mid \supp_{(A,w)}(P)\subseteq \mathcal{V}\}\,.
\]
\end{introthm}
Here the support of $P$ in $\perf(A,w)$, denoted $\supp_{(A,w)}(P),$ is the set of primes $\p$ where $P_\p\neq 0$ in $\perf(A_\p,w_\p)$. 
In the case $w$ itself is zero, \cref{intro_thm_2} provides a classification of the thick tensor ideals of $\perf(A,0)$; this is \cref{cor_tt}, and recovers the theorem of Hopkins--Neeman--Thomason for (graded) rings. Furthermore, this result can be regarded as the  algebraic analog of the classification of Hirano in \cite[Theorem~1.1]{Hirano}. 

More generally, the classification in \cref{intro_thm_2} establishes that $\perf(A,w)$ has a tensor module generator whenever $\supp(A,w)$ is closed (which holds, for example,  when $w=0$ or when $w$ is a regular element and $A$ is a smooth over a graded field); this is the content of \cref{generators}. When $A$ is regular, thick subcategories of $\perf(A,w)$ are thick tensor submodules and so the classification in \cref{intro_thm_2}, specializes to the bijection applied in \cref{main_result}, with the same maps $\sigma$ and $\tau$ defined above:
\[
\begin{tikzcd}
\biggm\{\begin{matrix}\text{Thick subcategories}\\ \text{of } \perf(A,w) \end{matrix} \biggm\}\arrow[r,shift left=0.8ex,"\sigma"]&\biggm\{\begin{matrix}\text{Specialization closed subsets}\\ \text{of } \supp(A,w)\end{matrix} \biggm\}\arrow[l,shift left=0.8ex,"\tau"]
\end{tikzcd}\,.
\]

This last result was also recently proven by Briggs in \cite[Theorem~A.11]{Briggs-Walker}; interestingly, the proof there deduces the result from a theorem that avoids the use of tensor actions, but instead requires certain internal Hom's to be perfect (the condition being satisfied when $A$ is regular, for instance); cf.\@ \cref{remark_ben}. 

In light of the recent works \cite{Lattices,Bartheletal,BIKPfiber}, it is natural to wonder if fiberwise classifications of thick subcategories can be used to achieve the classification in \cref{main_intro_thm},  especially since (in this setting) the fibers are exterior algebras and so the (classical) BGG correspondence could almost immediately apply. However, this strategy requires the use of certain ``big" categories (with a tensor action), and developing this setup will take us too far astray; see \cref{remark_ben_jian_josh}. Furthermore, it is worth commenting that part of the merit of this work is it avoids the use of infinite constructions to establish the classification in \cref{main_intro_thm}, similar to the strategy in \cite{CI} for artinian complete intersection rings. 

%%%%%%%%%%%%%%%%%%%%%%%%%%%%%%%%%%%%%%%%%%%%%%%%%%%%%%
\section{(Curved) dg modules}
%%%%%%%%%%%%%%%%%%%%%%%%%%%%%%%%%%%%%%%%%%%%%%%%%%%%%%
Throughout fix a base commutative ring $R$. We first review the necessary background on \textbf{differential graded} (=\textbf{dg}) algebras and modules.

\begin{chunk}
Recall that a dg $R$-algebra $A=\{A_i\}_{i\in \mathbb{Z}}$ is a graded $R$-algebra equipped with a degree $-1$ endomorphism $\del$ satisfying  $\del^2=0$ and the Leibniz rule:
\[
\del(a\cdot b)=\del(a)\cdot b+(-1)^{|a|}a\cdot \del(b)\,;
\]
when $a\neq 0$, the value $|a|$ is the unique integer $i$ such that $a$ belongs to $A_i.$  Similarly, a (left) graded $A$-module $M=\{M_i\}_{i\in \mathbb{Z}}$ is a (left) dg $A$-module if it is equipped with a square-zero degree $-1$ endomorphism $\del^M$ such that 
\[
\del^M(a\cdot m)=\del(a)\cdot m+(-1)^{|a|}a\cdot \del^M(m)\quad \text{for }a\in A\,, \ m\in M\,.
\]
For a dg $A$-module $M$, its homology $\h(M)=\{\h_i(M)\}_{i\in \mathbb{Z}}$ is a graded module over the graded $R$-algebra $\h(A)=\{\h_i(A)\}_{i\in \mathbb{Z}}$.
\end{chunk}

\begin{chunk}
Fix a dg $R$-algebra $A$, and let $M,N$ be dg $A$-modules; in this article, by a module we mean a \emph{left} module.  A morphism of dg $A$-modules $\alpha\colon M\to N$ is a morphism of $R$-complexes that is $A$-linear; it is a quasi-isomorphism (indicated with `$\simeq$') if the induced map of graded $\h(A)$-modules is an isomorphism. 
Let $\D(A)$ denote the derived category of dg $A$-modules, obtained by formally inverting quasi-isomorphisms of dg $A$-modules; cf.\@ \cite[Section~9]{Krause:2022}  or \cite[Section~2]{ABIM}. We will regard $\D(A)$ as a triangulated category in the standard way, with $\shift$ denoting the suspension functor. 
We write $\Df(A)$ for its full subcategory consisting of all dg $A$-modules $M$ where $\h(M)$ is a finitely generated graded $\h(A)$-module.
\end{chunk}

Let $\Cat$ be a triangulated category and recall that a full triangulated subcategory $\T$ of $\Cat$ is \textbf{thick} if it is closed under taking direct summands. Given an object $X$ in $\Cat$, then $\thick_{\Cat}(X)$ denotes the smallest thick subcategory $\Cat$ containing $X.$

\begin{example}
    Assume $A$ is a dg $R$-algebra with $\h(A)$ a noetherian graded algebra. In this case, $\Df(A)$ is a thick subcategory of $\D(A)$. 
\end{example}

\begin{example}\label{perf_dg}
    For a dg $R$-algebra $A$, we let $\perf(A)$ denote $\thick_{\D(A)}(A)$, and call an object $P$ in $\perf(A)$ a \textbf{perfect} dg $A$-module. These are the dg $A$-modules that are  quasi-isomorphic to a summand of a dg $A$-module $F$ where $F$ admits a filtration 
    \[
    0=F(-1)\subseteq F(0)\subseteq F(1)\subseteq \ldots \subseteq F(n)=F
    \]
   such that each $F(i)/F(i-1)$ is a finite sum of shifts of $A$; see \cite[Section~3]{ABIM}. Note that perfect dg $A$-modules are exactly the compact objects in $\D(A)$. That is to say, $M$ is a perfect dg $A$-module if and only if $
    \Hom_{\D(A)}(M,-)$ commutes with arbitrary direct sums; see \cite[Proposition~9.1.13]{Krause:2022}.
\end{example}

Next we recall the theory of curved dg modules over a curved algebra. For more on these objects see \cite{Posiselski:2011}. 

\begin{chunk}
    Let $A=\{A_i\}_{i\in \Z}$ be a graded commutative $R$-algebra, and fix $w\in A_{-2}$. A \textbf{curved dg $A$-module having curvature} $w$ is a graded $A$-module $M=\{M\}_{i\in \mathbb{Z}}$ equipped with a degree $-1$ map $\del^M\colon M\rightarrow M$  satisfying that 
\[
\del^M(a\cdot -)=(-1)^{|a|}a\del^M(-)\quad\text{and}\quad \del^M\del^M=w\cdot \id^M\,.
\] 
A morphism of curved dg $A$-modules is a morphism of graded $A$-modules $\alpha\colon M\to N$ such that $\alpha\del^M=\del^N\alpha$. 
\end{chunk}

\begin{example}
    Let $A$ be a graded commutative  $R$-algebra.  A curved dg $A$-module having curvature zero is exactly a dg $A$-module. 
\end{example}

\begin{example}
    A matrix factorization of a regular element $f$ in $R$, in the sense of Eisenbud~\cite{Eis}, can be regarded as a curved dg module. Indeed, let $A=Q[\chi^{\pm 1}]$ where $|\chi|=-2$ and consider a matrix factorization 
    \[
    P= P_0\xra{d_0}P_1\xra{d_1} P_0\,;
    \]
    that is, each $P_i$ is a projective $R$-module and $d_0d_1=f\cdot \id^{P_1}$ and $d_1d_0=f\cdot \id^{P_0}$. One can associate to $P$ the curved dg $A$-module $\tilde{P}=\shift A\otimes_Q P_0\oplus A\otimes_Q P_1$ with 
    \[
    \del^{\tilde{P}}=\begin{pmatrix}
        0 & \chi\otimes d_1 \\
        1\otimes d_0 & 0 
    \end{pmatrix}
    \]
   where the curvature of $\tilde{P}$ is $f\chi\in A_{-2}$.  Here $\tilde{P}$ can be viewed as the two-periodic sequence 
    \[
    \cdots \to P_0\chi^{-1}\xra{d_0}P_1\chi^{-1}\xra{d_1}P_0\xra{d_0}P_1\xra{d_1}P_0\chi\xra{d_0}P_1\chi \to \cdots \,.
    \]
\end{example}

\begin{chunk}\label{def_curved_dg}
Fix a graded commutative $R$-algebra $A$. 
We consider the category of curved dg $A$-modules with curvature $w$ whose underlying graded $A$-module is a finitely generated graded projective $A$-module; this category is denoted $\proj(A,w).$ This is a Frobenius exact category with injective/projective objects summands of $\cone(\id^P)$ with $P$ in $\proj(A,w)$, the obvious admissible short exact sequences, and suspension $\shift P$ the object in $\proj(A,w)$ satisfying 
\[
(\shift P)_n=P_{n-1}\quad\text{and}\quad a\cdot \shift p=(-1)^{|a|}\shift(a\cdot p)\quad\text{for all }a\in A\,, \ p\in P\,.
\]
Let $\perf(A,w)$ denote the stable category of $\proj(A,w)$, which then obtains a triangulated structure in the usual way, called the category of \textbf{perfect curved dg $A$-modules having curvature $w$}. Note that $\perf(A,w)$ is also the homotopy category of $\proj(A,w)$ regarded as a dg category. 
\end{chunk}

\begin{chunk}\label{tensor_actions}
       Fix $A=\{A_i\}_{i\in\mathbb{Z}}$,  a graded commutative noetherian ring. For $v,w\in A_{-2}$ there is a pairing 
       \[
       \proj(A,v)\times \proj(A,w)\to \proj(A,v+w)
       \]
       assigning $(P,P')$ to the curved dg $A$-module with curvature $v+w$ given by 
       \[
       P\otimes_A P'\quad \text{with}\quad  \del^{P\otimes_A P'}(x\otimes y)=\del^P(x)\otimes y+(-1)^{|x|}x\otimes \del^{P'}(y)\quad\text{for } x\in P,y\in P'\,.
       \]
       As $-\otimes_A-$ is bi-exact on projective $A$-modules and bi-additive, it follows that there are induced pairings 
       \[
       \perf(A,v)\times \perf(A,w)\to \perf(A,v+w)\,,
       \]
       that are bi-exact, bi-additive and suitably compatible. 
       \end{chunk}

\begin{chunk}
    Paralleling the discussion above in \cref{tensor_actions}, there are also bi-exact, bi-additive functors \[\perf(A,v)^{\text{op}}\times \perf(A,w)\to \perf(A,w-v)\]
       where $(P,P')\mapsto \Hom_A(P,P')$ with the usual Hom-differential 
       \[\del^{\Hom_A(P,P')}(f)=\del^{P'}f-(-1)^{|f|}f\del^P\,.\]
       For  $P$ in $\perf(A,v)$, we write $P^\vee=\Hom_A(P,A)$  which belongs to $\perf(A,-v)$. 
\end{chunk}

\begin{remark}\label{dg_cat}
     Note that $\perf(A,w)$ is the homotopy category of $\proj(A,w)$ regarded as a pre-triangulated dg category. In particular, \[
     \Hom_{\perf(A,w)}(P,P')=\h(\Hom_A(P,P'))=\frac{\ker \del^{\Hom_A(P,P')}}{\Ima \del^{\Hom_A(P,P')}}
     \]
     for any pair $P,P'$ in $\perf(A,w)$. Furthermore, for $P$ in $\perf(A,v)$  and $P'$ in $\perf(A,w)$, then there are natural isomorphisms:
     \begin{enumerate}
         \item $\Hom_{\perf(A,v+w)}(P\otimes_A P', -)\cong \Hom_{\perf(A,w)}(P', \Hom_A(P,-))$, and
         \item $\Hom_A(-,P'\otimes_A P)\cong \Hom_A(-,P')\otimes_A P$ in $\perf(A,0)$.
     \end{enumerate}
\end{remark}

We will call a graded commutative noetherian ring \textbf{regular} if every finitely generated graded module has finite projective dimension. 

\begin{proposition}\label{p:perf}
    Let $A$ be a graded commutative noetherian ring. The obvious functor $\proj(A,0)\to \D(A)$ induces an exact functor $\perf(A,0)\to \D(A)$; if $A$ is regular or concentrated in nonnegative degrees, then the functor restricts to an equivalence of triangulated categories $\Phi\colon \perf(A,0)\to \perf(A)$. In particular, if $A$ is regular or concentrated in nonnegative degrees, then \[\thick_{\perf(A,0)}(A)=\perf(A,0)\,.\] 
\end{proposition}
\begin{proof}
Note that a curved dg $A$-module with curvature 0 is exactly a dg $A$-module. Also, the suspensions are the same on $\proj(A,0)$ and $\D(A)$, exact sequences in $\proj(A,0)$ are sent to triangles in $\D(A)$, and phantom maps in $\proj(A,0)$ are null-homotopic (in the classical sense), and hence, sent to zero maps in $\D(A)$. Therefore, we obtain the triangulated functor  $ \perf(A,0)\to \D(A)$. 

When $A$ is concentrated in nonnegative degrees, the fact that $\Phi$ is an equivalence follows from \cite[Proposition~4.4]{ABIM}.

Now assume that $A$ is regular. In this setting, any graded projective $P$ is semiprojective (in the usual sense that $\Hom_A(P,-)$ is exact and preserves quasi-isomorphisms); cf. \@ \cite[Proposition 1.2.8]{Po}.  Furthermore, as $A$ is noetherian any semiprojective dg $A$-module $P$, that is finitely generated as a graded $A$-module, must be a perfect dg $A$-module; indeed, $\Hom_{\D(A)}(P,-)$ commutes with arbitrary direct sums in this case and so by \cref{perf_dg} it follows that $P$ is in $\perf(A)$. 

As a consequence of the remarks above,  the image of each object $\perf(A,0)$ under this functor is in fact a perfect dg $A$-module. Thus, we obtain the induced exact functor $\Phi\colon \perf(A,0)\to \perf(A)$. It is clear $\Phi$ is essentially surjective; see \cref{perf_dg}. Finally, a pair of perfect dg $A$-modules are quasi-isomorphic if and only if they are homotopy equivalent, and hence, $\Phi$ is an equivalence. Furthermore, the equality now follows immediately, as $\Phi(A)=A$ and $\perf(A)=\thick_{\D(A)}(A)$.
\end{proof}
\begin{remark}
    In general, the image of the functor $\perf(A,0)\rightarrow \D(A)$ is not contained in $\perf(A)$; see Example \ref{perfct curved not perfect}.
\end{remark}

%%%%%%%%%%%%%%%%%%%%%%%%%%%%%%%%%%%%%%%%%%%%%%%%%%%%%%
\section{Tensor modules and support}
%%%%%%%%%%%%%%%%%%%%%%%%%%%%%%%%%%%%%%%%%%%%%%%%%%%%%%

Fix a graded commutative noetherian ring $A=\{A_i\}_{i\in\mathbb{Z}}$. 
First we recall how $\perf(A,0)$ is a tensor triangulated category~\cite{Balmer:2005}. 

\begin{example}\label{e_tt}
    By \cref{tensor_actions}, in the notation of \cref{def_curved_dg}, $\perf(A,0)$ is a tensor triangulated category. Explicitly, the tensor product $P\otimes_A P'$ is the underlying tensor product of projective graded $A$-modules  $P\otimes_A P'$ with \[\del^{P\otimes_A P'}(x\otimes y)=\del^P(x)\otimes y+(-1)^{|x|}x\otimes \del^{P'}(y)\] for $x\in P$ and $y\in P'$. Note that $A$ is the tensor unit of $\perf(A,0)$. 
    
    Now assume that $A$ is regular. The equivalence $\perf(A,0)\equiv \perf(A)$ of triangulated categories from \cref{p:perf} clearly upgrades to one of tensor-triangulated categories, where the category on the right is given the derived tensor product $-\lotimes_A-$. Therefore, by \cref{p:perf}, in this setting the tensor unit $A$ in $\perf(A,0)$ is a \emph{classical generator} of $\perf(A,0)$ in the sense that 
    \[\thick_{\perf(A,0)}(A)=\perf(A,0)\,.\] 
\end{example}

\begin{chunk}\label{tensor_module}
Fix a tensor triangulated category  $(\T,\otimes,\mathsf{1})$, and
 suppose that $\T$  acts on a triangulated category $\sf K$ in the sense of Stevenson~\cite[Section~3]{Stevenson:2013a}; see, also, \cite{Stevenson:2014b,Stevenson:2018a}. Roughly speaking, if we write $\odot$ to denote this action, then $\odot\colon \T\times \sf K\to \sf K$ is a bi-exact, bi-additive functor that satisfies certain compatibility relations up to natural isomorphisms like associativity and
\[
\mathsf{1}\odot x\simeq x\quad\text{for all }x\text{ in }\mathsf{K}\,.
\]
We call $\sf K$ a tensor module of $\T$, and tensor submodules of $\sf K$ are exactly the thick subcategories of $\sf K$ that are closed under the $\T$-action. Furthermore, if $\T=\thick_{\T}(\mathsf{1})$ then every thick subcategory of $\sf K$ is necessarily a tensor submodule of $\sf K$. 

Given an object $x$ (or a set of objects) in $\sf K$, we denote $\thick^{\odot}_{\sf K}(x)$ for the smallest thick subcategory of $\sf K$ that is closed under the action of $\T$. Therefore, whenever $\T=\thick_{\T}(\mathsf{1})$ one has $\thick^{\odot}_{\sf K}(x)=\thick_{\sf K}(x)$ for all $x$ in $\sf K$. 
\end{chunk}

\begin{example}
Fix $w\in A_{-2}$, and recall that $\perf(A,0)$ is a tensor triangulated category by \cref{e_tt}. By \cref{tensor_actions}, we obtain that $\perf(A,w)$ is a tensor $\perf(A,0)$-module. We let $\odot$ denote the action of $\perf(A,0)$ on $\perf(A,w)$, and write 
       \[
       \thick_{(A,w)}^{\odot}(P)=\thick_{\perf(A,w)}^{\odot}(P)
       \]
       where $P$ is in $\perf(A,w)$ (or is a set of objects from $\perf(A,w)$). 
\end{example}

\begin{example}\label{perfct curved not perfect}
     Assume $A=k[y,z]/(y^2,yz,z^2)$, where $|y|=0, |z|=-2$, and $k$ is a field. Let $M$ denote the dg $A$-module whose underlying graded $A$-module is
 $A\oplus \shift^{-1}A$ and differential given by $\del^M=\begin{pmatrix}
         0 & z\\
y& 0
     \end{pmatrix}
     $.
    By definition, $M$ is a perfect curved dg $A$-module with curvature $0$, i.e., $M$ is in  $\perf(A,0)$. However, as a dg $A$-module, $M$ is not perfect; see \cite[Example 4.12.2]{ABIM}.

   Moreover, $\thick_{(A,0)}(A)\neq \thick_{(A,0)}^{\odot}(A)$, as $\thick_{(A,0)}(A)$ is not a thick tensor $\perf(A,0)$-submodule of $\perf(A,0)$.
This follows from the fact that 
$M\otimes_A A\cong M$ does not belong to $\thick_{\perf(A,0)}(A):$ indeed, if $M$ is in  $\thick_{(A,0)}(A)$,  
then, by applying the exact functor $\perf(A,0)\rightarrow \D(A)$ it follow that $M$ belongs to $\thick_{\D(A)}(A)$; see \cite[Proposition 3.7]{ABIM}. This contradicts with that $M$ is not a perfect dg $A$-module.
\end{example}

For the remainder of this section, fix $w\in A_{-2}$. 

\begin{chunk}\label{curved_spec}
Let $\spec(A)$ denote the collection of homogeneous primes of $A$. For a homogeneous prime $\p$ in $\spec(A)$, we write $\kappa(\p)$ for the graded field $A_\p/\p A_\p$ and set $w(\p)=\kappa(\p)\otimes_A w$ in $\kappa(\p)$. There is an exact functor 
\[
\perf(A,w)\to \perf(\kappa(\p),w(\p))
\]
induced by assigning to each $P$ in $\perf(A,w)$ the perfect curved dg $\kappa(\p)$-module $P(\p)\coloneqq \kappa(\p)\otimes_A P=\kappa(\p)\odot P$ in $\perf(\kappa(\p),w(\p))$.

When $w\in \p$, then $w(\p)=0$ and since $\kappa(\p)$ is a graded field it follows that \[\perf(\kappa(\p),w(\p))=\perf(\kappa(\p),0)\equiv \perf(\kappa(p))\,;\]
cf.\@ \cref{p:perf}. In particular, every object in $\perf(\kappa(\p),w(\p))$ is a direct sum of shifts of copies of $\kappa(\p)$.

If $w\notin \p$, then $w(\p)$ is a unit  in the graded field $\kappa(\p)$.  In the case that $\kappa(\p)$ is concentrated in even degrees it follows that $\perf(\kappa(\p),w(\p))\equiv 0.$  Indeed, any object $T$ in $\perf(\kappa(\p),w(\p))$ is of the form 
\[
T=T_{\textrm{even}}\oplus T_{\textrm{odd}} \quad\text{and}\quad \del^T=\begin{pmatrix} 0 & d_0\\
    d_1 & 0
    \end{pmatrix}
    \]
    with $d_0d_1=w(\p)\cdot \id^{T_{\textrm{even}}}$ and $d_1d_0=w(\p)\cdot \id^{T_{\textrm{odd}}}$. Since $w(\p)$ is a unit and $\kappa(\p)$ is concentrated in even degrees, $T$ is zero in $\perf(\kappa(\p),w(\p))$ as 
    \[
    \begin{pmatrix} 0 &  d_0w(\p)^{-1}\\
    0 & 0
    \end{pmatrix} 
    \]
    is a null-homotopy for $\id^T$. 
\end{chunk}

\begin{definition}
    Let $P$ be in $\perf(A,w)$. Define the \textbf{support} of $P$ to be 
    \[
    \supp_{(A,w)}(P)=\{\p\in \spec(A): P_\p\neq 0\text{ in }\perf(A_\p,w_\p)\}\,.
    \]
\end{definition}

In the previous definition, $P_\p$ is the homogeneous localization of the graded $A$-module $P$ at the homogeneous prime ideal $\p$; see \cite[Section~1.5]{Bruns/Herzog:1998}. Furthermore, when we write $\supp_A(M)$, for a graded $A$-module $M$, we mean the ordinary support of $M$. That is to say,
    \[
    \supp_A(M)=\{\p\in \spec(A): M_\p\neq 0\}\,.
    \]
By Nakayama's lemma, one also has the following interpretation of support of a perfect curved dg $A$-module. 
    
\begin{lemma}\label{fg_closed}
    For $P$ in $\perf(A,w)$  one has \[\{\p\in \spec(A): P(\p)\neq 0\text{ in }\perf(\kappa(\p),w(\p))\}=\supp_{(A,w)}(P)=\supp_A(A\cdot \id^P)\,.
    \]
    In particular, $\supp_{(A,w)}(P)$ is a closed subset of $\spec(A)$, and $\supp_{(A,w)}(P)=\varnothing$ if and only if $P=0$ in $\perf(A,w)$.
\end{lemma}
\begin{proof}
    As $\id^P_\p\cong \id^{P_\p}$ the second equality holds. 
    Also, if $P_\p=0$, then $P(\p)=0$. Therefore, it suffices to show if $P(\p)=0$, then $P_\p=0$. 

    To this end, there exists $\beta\colon P(\p)\to P(\p)$ a degree 1 map such that 
\[\del^{P(\p)}\beta+\beta\del^{P(\p)}=\id^{P(\p)}\,.
    \]
    Lift $\beta$ to a degree 1 map $\alpha\colon P_\p\to P_\p$ satisfying $\alpha(\p)=\beta$. Now observe that 
    \[
   \alpha'\coloneqq \alpha\del^{P_\p}+ \del^{P_\p}\alpha\colon P_\p\to P_\p
    \]
    is an endomorphism of the curved dg $A_\p$-module $P_\p,$ and by construction 
    \[
    \alpha'= 0\quad\text{in }\perf(A_\p,w_\p)\,.
    \]
    However, $\alpha'(\p)=\id^{P(\p)}$ and so by Nakayama's lemma it follows that $\alpha'$ is an automorphism of $P_\p.$ Combining this with the fact that $\alpha'=0$ in $\perf(A_\p,w_\p)$, it follows that $P_\p=0$ in $\perf(A_\p,w_\p)$.
\end{proof}

\begin{chunk}    
    \label{support}
    The following are elementary properties of support. 
    \begin{enumerate}
        \item $\supp_{(A,w)}(P)=\supp_{(A,w)}(\shift P)$;
    \item $\supp_{(A,w)}(P\oplus P')=\supp_{(A,w)}(P)\cup \supp_{(A,w)}(P')$;
    \item if $P\to P'\to P''\to \shift P$ is an exact triangle in $\perf(A,w)$, then 
    \[
    \supp_{(A,w)}(P)\subseteq \supp_{(A,w)}(P')\cup \supp_{(A,w)}(P'')\,;
    \]
    \item for any $P$ in $\perf(A,w)$ and $P'$ in $\perf(A,v)$, there is a containment \[\supp_{(A,w+v)}(P\otimes_A P')\subseteq \supp_{(A,w)}(P)\cap \supp_{(A,v)}(P')\,.\]
    \end{enumerate}
        By \cref{curved_spec}, when $A$ is concentrated in even degrees it follows that \[\supp_{(A,w)}(P)\subseteq \mathcal{V}(w)=\{\p\in \spec(A): w\in \p\}\,.\] Furthermore, in this setting, it also follows that equality holds in (4), above. Indeed,
    \[
    P\otimes_A P' (\p) \cong P(\p)\otimes_{\kappa(\p)}P'(\p)
    \]
    in $\perf(\kappa(\p),0)$ and so 
    \[
    P(\p)\otimes_{\kappa(\p)}P'(\p)= 0\iff P(\p)=0 \ \text{ or } \ P'(\p)=0\,.
    \]
    Therefore, it suffices to refer to \cref{fg_closed}. 
    \end{chunk}

The next lemma is a curved version of the tensor-nilpotence theorem of Hopkins--Neeman; see \cite[Theorem~10]{Hopkins} and \cite[Theorem~1.1]{Neeman92}; see also \cite[Theorem~3.3]{CI}. The version we state here is inspired by (and is a slight generalization of) a version of the tensor-nilpotence theorem for matrix factorizations due to Hirano~\cite[Lemma~4.2]{Hirano}. 

\begin{lemma}[Curved tensor-nilpotence]\label{l_tensor_nilpotence}
Assume $A$ is a graded commutative noetherian ring, with $w\in A_{-2}$, and fix $P,P'$ in $\perf(A,w)$.  If $\alpha\colon P'\to P$ is a map in $\perf(A,w)$ with $\alpha(\p)=0$ for all $\p\in \supp_{(A,w)}(P)$, then there exists a positive integer $n$ such that $\alpha^{\otimes_A n}=0$ in $\perf(A,n\cdot w)$.
\end{lemma}

\begin{proof}
    By \cref{dg_cat}, there is a natural isomorphism 
    \[
    \Phi\colon \Hom_{\perf(A,w)}(P',P)\xra{\cong} \Hom_{\perf(A,0)}(A,\Hom_A(P',P))
    \]
    where $\Phi(\alpha^{\otimes_A n})$ is identified with $\Phi(\alpha)^{\otimes_An}.$ Therefore, we have reduced to the case that $w=0$ and $P'=A$. Now it remains to note that the desired tensor nilpotence was shown to hold in $\perf(A)$ in \cite[Theorem~3.3]{CI}, but the proof shows tensor nilpotence holds in $\perf(A,0)$; namely, the argument in loc.\@ cit.\@  following their ``\emph{Step 2}" holds for objects in $\perf(A,0)$, and establishes the desired statement. 
\end{proof}

\begin{lemma}
\label{l_nilpotent_map}
Let $\alpha\colon F\to G$ be in $\perf(A,0)$ and $P$ in $\perf(A,w)$. If $\alpha(\p)=0$ for all $\p\in \supp_{(A,w)}(P)$, then there exists $n>0$ such that $\alpha^{\otimes n}\otimes \id^P=0$ in $\perf(A,w).$
\end{lemma}
\begin{proof}
    Below we write $\otimes$ for $\otimes_A$. First observe that the map $\alpha\otimes \id^P\colon F\otimes P\to G\otimes P$ in $\perf(A,w)$ satisfies 
    \[
    \alpha\otimes \id^P\otimes \kappa(\p)\cong \alpha\otimes \kappa(\p)\otimes \id^P=0\otimes \id^P=0
    \]
    for all $\p\in \supp_{(A,w)}(P).$ Hence, by \cref{l_tensor_nilpotence}, we have that there is a positive integer $n$ where 
    \[
    (\alpha\otimes \id^P)^{\otimes n}\cong \alpha^{\otimes n}\otimes (\id^{P})^{\otimes n}=0 \quad\text{in}\quad \perf(A,n\cdot w)\,. 
    \]

    Now observe the desired result follows from the claim that $\alpha^{\otimes i}\otimes \id^P$ is a retract of $\alpha^{\otimes i}\otimes (\id^{P^\vee})^{\otimes i}\otimes (\id^P)^{\otimes (i+1)}$ for all nonnegative integers $i$, as the latter is zero when $i\gg 0$ by what we have shown above. The claim follows easily by induction, and so it suffices to check the base case. Here we have that the composition 
    \[
    P\to \Hom_A(P,P)\otimes P\cong P\otimes P^\vee\otimes P \to P
    \]
    is the identity, where the first map is given by $x\mapsto \id^P\otimes x$, the second map is the isomorphism in \cref{dg_cat}(2), and the third map is the evaluation map on the second two factors.
\end{proof}

The upshot of the lemmas above is the following theorem, which is a curved version of \cite[Proposition~5.5]{Hirano}; the proof of \cref{l_support_containment} follows a similar argument to that in \emph{loc}.\@ \emph{cit}.  

\begin{theorem}
    \label{l_support_containment}
  Assume $A$ is a graded commutative noetherian ring, with $w\in A_{-2}$. If $P,P'$ are in $\perf(A,w)$, then 
    \[P'\in \thick^{\odot}_{(A,w)}(P)\quad \text{if and only if}\quad\supp_{(A,w)}(P')\subseteq \supp_{(A,w)}(P)\,.\]
\end{theorem}
\begin{proof}
    The forward implication follows easily from the facts on support listed in \cref{support}.  For the converse, complete the morphism $A\to \Hom_A(P,P)$ in $\perf(A,0)$ to an exact triangle  
    \[
    F\xra{\alpha} A\to \Hom_A(P,P)\to \shift F \,.
    \]
    We claim that $\alpha^{\otimes n}\otimes \id^{P'}$ is zero for some positive integer $n$. Indeed, the induced maps \[\kappa(\p)\to \Hom_{\kappa(\p)}(P(\p),P(\p))\] are nonzero split monomorphisms for each $\p\in \supp_{(A,w)}(P')$, and hence $\alpha(\p)=0$ for all such $\p$. Thus, by \cref{l_nilpotent_map}, we have that $\alpha^{\otimes n}\otimes \id^{P'}=0$ in $\perf(A,w)$. In particular, the triangle 
    \[
    F^{\otimes n}\otimes P' \xra{\alpha^{\otimes n}\otimes \id^{P'}=0}P' \to \cone(\alpha^{\otimes n})\otimes P'\to \shift F^{\otimes n}\otimes P' \quad\text{in }\perf(A,w) 
    \]
    says that  
    $P'$ is a retract of $\cone(\alpha^{\otimes n})\otimes P'$. 
    
    It is now a formal argument that $\cone(\alpha^{\otimes n})\otimes P'$ belongs to $\thick_{(A,w)}^{\odot}(P)$, which would finish the proof; we include the argument here for the sake of completeness. For each $i\geqslant 0$, define \[C(i)\coloneqq\cone(F^{\otimes i}\xra{\alpha^{\otimes i}} A^{\otimes i}\cong A)\,.\]
    Now iteratively applying the octahedral axiom we obtain a diagram with exact rows and columns in $\perf(A,0)$:
 \begin{center}
    \begin{tikzcd}
 F^{\otimes (i+1)} \arrow[d,"="] \arrow[r,"\alpha^{\otimes i}\otimes \id^F"] & F\arrow[d]\arrow[r]& C(i)\otimes F\arrow[r] \ar[d] & \shift F^{\otimes (i+1)} \arrow[d,"="] \\
       F^{\otimes(i+1)}  \arrow[r,"\alpha^{\otimes(i+1)}"] & A\arrow[r]\arrow[d]& C(i+1)\arrow[r] \ar[d] & \shift  F^{\otimes (i+1)} \\
       & C(1) \arrow[d] \ar[r,"="] & C(1) \ar[d] &\\ 
 & \shift F\ar[r]&\shift C(i)\otimes F &
    \end{tikzcd} 
\end{center}
In particular, there is an exact triangle 
\[
C(i)\otimes F\otimes P'\to C(i+1)\otimes P'\to C(1)\otimes P'\to \shift C(i)\otimes F\otimes P'
\]
and so by induction it would suffice to show $C(1)\otimes P'\in \thick_{(A,w)}^{\odot}(P).$ However, this is obvious as 
\[
C(1)\otimes P'=\Hom_A(P,P)\otimes P'\cong P^\vee\otimes P\otimes P'\cong P\otimes (P^\vee\otimes P')\,,
\]
where the first isomorphism uses \cref{dg_cat}(2). 
\end{proof}

%%%%%%%%%%%%%%%%%%%%%%%%%%%%%%%%%%%%%%%%%%%%%%%%%%%%%%
\section{Classification of thick tensor submodules}
%%%%%%%%%%%%%%%%%%%%%%%%%%%%%%%%%%%%%%%%%%%%%%%%%%%%%%
In this section, fix a graded commutative noetherian ring $A$ with $w\in A_{-2}$. 

\begin{definition}\label{defn_sing}
    Define the \textbf{support} of $(A,w)$ to be 
    \begin{align*}
    \supp(A,w)&\coloneqq \bigcup_{P\in \perf(A,w)}\supp_{(A,w)}(P) \\
    &=\{\p\in \spec(A): \perf(A,w)\to\perf(A_\p,w_\p)\text{ is nonzero}\}\,.
    \end{align*}
\end{definition}

The next result bounds the support of $(A,w)$ in elementary terms when $A$ is concentrated in even degrees. 
\begin{proposition}\label{sing_locus}
     Assume $A$ is a graded commutative noetherian ring concentrated in even degrees, and fix $w\in A_{-2}$.   There are inclusions \[
   \{\p\in \spec(A): w\in \p^2\}\subseteq  \supp(A,w)\subseteq  \{\p\in \spec(A): w_\p\in \p^2 A_\p \}\,.\] In particular, $\supp(A,0)=\spec(A)$.
 
 Furthermore, if, in addition, $A$ is assumed to be regular, then
   \[\supp(A,w)=  \{\p\in \spec(A): w_\p\in \p^2 A_\p \}\,.\]
\end{proposition}
\begin{proof}
    Let $\p\in \supp(A,w)$. Then by definition there exists  $P$ in $\perf(A,w)$ with $P_\p\neq 0$ and $w\in \p$. Since $A$ is concentrated in even degrees, then 
    \[P=P_{\mathrm{even}}\oplus P_{\mathrm{odd}} \quad\text{with}\quad \del^{P}=\begin{pmatrix} 0 & d_1\\
    d_0 & 0
    \end{pmatrix}\,.\]
    Since $A_\p$ is local, using elementary column and row operations, $P_\p\cong T\oplus T'\oplus P'$
    where  $T=T_{\textrm{even}}\oplus T_{\textrm{odd}}$, $T'=T'_{\textrm{even}}\oplus T'_{\textrm{odd}}$,  $P'=P'_{\textrm{even}}\oplus P'_{\textrm{odd}}$, and 
\[
\del^T=\begin{pmatrix} 0 & u^{-1}w\\
    u & 0
    \end{pmatrix}\,, \quad \del^{T'}=\begin{pmatrix} 0 & (u')^{-1}\\
    u'w & 0    \end{pmatrix}\,, \quad \text{and }\del^{P'}=\begin{pmatrix} 0 & d_1'\\
    d_0'& 0
    \end{pmatrix}
\]
where $u$, $u'$ are square matrices with units in $A_\p$ on their respective diagonals, and $\del^{P'}(P')\subseteq \p P'$. Note that $T$ and $T'$ are zero in $\perf(A_\p,w_\p)$. Since $P\neq 0$, we have that $P'\neq 0$. 
Therefore, 
    \[
    w_\p\cdot \id^{P'}=\begin{pmatrix} 0 & d_1'\\
    d_0' & 0
    \end{pmatrix}^2=\begin{pmatrix}d_1'd_0'&0 \\
    0 & d_0'd_1'
    \end{pmatrix}\in \p^2P'\,
    \]
    and so $w_\p\in \p^2A_\p.$ 

 Next, if $w\in \p^2$ then we can write $w_\p=\sum_{i=1}^n a_ib_i$ with $a_i,b_i\in \p$.  For each $i$, set $P(i)$ to be the object in $\perf(A, a_ib_i)$ given by 
    \[
    A\oplus  \shift^{|a_i|+1} A \quad\text{with}\quad  \del^{P(i)}=\begin{pmatrix}
    0 & a_i\\ 
    b_i & 0
\end{pmatrix}\,.
    \]
 Therefore, 
    \[
    P\coloneqq P(1)\otimes_A \cdots \otimes_A P(n) \quad \text{is in}\quad \perf(A,w)\,.
    \] 
    As the entries in the differential of $P_\p$ are all non-units in $A_\p$ it follows that $P_\p$ cannot be zero in $\perf(A_\p,w_\p)$. Therefore, by definition $\p\in \supp(A,w).$ 

   For the furthermore statement, the same argument as in the proof of \cite[Proposition~3.5]{Hirano} (see also \cite[Theorem A.11]{Briggs-Walker}) applies; we briefly sketch it for the convenience of the reader. 
   
   One first reduces to the case that $A$ is local (in the graded sense), and assume $\p\in \spec A$ satisfies that $w_\p\in\p^2A_\p$. If $w_\p=0$, then $w=0$ as $A$ is a domain, and hence $\p\in \supp(A,w)$ by the already established containment above. If $w_\p\neq 0$, then $w$ is $A$-regular, again making use of the fact that $A$ is a domain. In this setting, $\perf(A,w)$ is equivalent to the singularity category of the graded hypersurface ring $A/(w)$; cf.\@ \cite{Buch,Orlov}.
In particular, $\supp(A,w)$ corresponds to the singular locus of $A/(w)$, that is to say, 
 \[
 \supp(A,w)\cong\{\p\in \spec A/(w): (A/(w))_\p\text{ is not regular}\}\,.
 \]
 Since $w\neq 0$ and $w_\p\in \p^2 A_\p$, we get that $A_\p/ w_\p A_\p\cong (A/(w))_\p$ is not regular. Thus, $\p\in \supp(A,w)$. This completes the proof.
\end{proof}

\begin{remark}
    The construction of $P$ in the second paragraph of the proof of \cref{sing_locus} is based on a classical one; cf.\@ \cite{BGS}. The same construction appears in the proof of \cite[Lemma 3.6]{Hirano}; see also \cite[Theorem~A.11]{Briggs-Walker}. Furthermore, similar constructions are present also in \cite{Avramov/Gasharov/Peeva:1997,Yoshino}.
\end{remark}

\begin{remark}
 As noted in \cref{sing_locus}, when $w=0$ we have that $\supp(A,w)$ is simply $\spec(A)$, and hence it is clearly closed. In an orthogonal direction, when $A$ is regular and $w$ is a regular element, as noted in the proof of \cref{sing_locus}, $\supp(A,w)$ can be identified with the singular points 
 \[
 \{\p\in \spec A/(w): (A/(w))_\p\text{ is not }regular\}\,;
 \]
 in the case that $A$ is further assumed to be smooth over a graded field it follows that $\supp(A,w)$ is closed in $\spec(A)$, as $A/(w)$ is excellent. 
 
 A set theoretic description of $\supp(A,w)$ is given when $A$ is a symmetric algebra over a regular ring in \cref{decomposition}. It would be interesting to describe exactly when $\supp(A,w)$ is closed in elementary terms like the conditions bounding this set in \cref{sing_locus}.
\end{remark}

Recall that a union of closed subsets is called specialization closed. 

\begin{theorem}\label{classification}
Assume $A$ is a graded commutative noetherian ring and $w\in A_{-2}$.  If $A$ is concentrated in even degrees, then there are inclusion preserving bijections 
\[
\begin{tikzcd}
\biggm\{\begin{matrix}\text{Thick tensor submodules}\\ \text{of } \perf(A,w) \end{matrix} \biggm\}\arrow[r,shift left=0.8ex,"\sigma"]&\biggm\{\begin{matrix}\text{Specialization closed subsets}\\ \text{of } \supp(A,w)\end{matrix} \biggm\}\arrow[l,shift left=0.8ex,"\tau"]
\end{tikzcd}\,,
\]
where the inverse bijections are given by
\[
\sigma(\T)=\bigcup_{P\in \T}\supp_{(A,w)}(P)\quad\text{and}\quad \tau(\mathcal{V})=\{P\in \perf(A,w)\mid \supp_{(A,w)}(P)\subseteq \mathcal{V}\}\,.
\]
\end{theorem}
\begin{proof}
    By \cref{support} and \cref{fg_closed}, the maps $\sigma$ and $\tau$ are well-defined, respectively. Also, it is clear that the maps are inclusion preserving and satisfy: for any specialization closed subset $\mathcal{V}$ of $\supp(A,w)$ we have $\sigma\tau(\mathcal{V})\subseteq \mathcal{V}$, and for any thick tensor submodule $\T$ of $\perf(A,w)$ we have $\T\subseteq \tau\sigma(\T).$ Therefore, it remains to show the two reverse containments. 

Let $\p\in \mathcal{V}$. By \cref{fg_closed}, there exists $P\in \perf(A,w)$ such that 
\[P(\p)\neq 0\quad\text{in }\perf(\kappa(\p),w(\p))\,.\] Let $\p=(x_1,\ldots,x_n)$ and let $K(i)$ denote 
\[
A\oplus \shift^{|x_i|+1} A \quad\text{with}\quad \del^{K(i)}=\begin{pmatrix}
    0 & 0\\ 
    x_i & 0
\end{pmatrix},
\]
where $1\leqslant i\leqslant n$. 
Set $K=K(1)\otimes_A\cdots \otimes_A K(n)$ which belongs to $ \perf(A,0)$ and satisfies $\supp_{(A,0)}(K)=\mathcal{V}(\p)$. Therefore,  $K\odot P$ is in $\perf(A,w)$ and 
\[\supp_{(A,w)}(K\odot P)= \supp_{(A,0)}(K)\cap \supp_{(A,w)}(P)=\mathcal{V}(\p)\,;\]
the first equality holds by \cref{support}, and the second equality uses that $\supp_{(A,0)}(K)=\mathcal{V}(\p)$ and $\p\in \supp_{(A,w)}(P)$
 with $\supp_{(A,w)}(P)$ being closed (see \cref{fg_closed}). Thus, we have shown $\p\in \supp_{(A,w)}(K\odot P)\subseteq \sigma\tau(\mathcal{V})$.

 Let $P'$ be in $\tau\sigma(\T)$. Then by definition, 
 \[
 \supp_{(A,w)}(P')\subseteq \bigcup_{P\in \T }\supp_{(A,w)}(P)\,,
 \]
 and now using that $\supp_{(A,w)}(P')$ is closed (by \cref{fg_closed}) and $\spec(A)$ is noetherian, there exist $P(1),\ldots,P(n)$ in $\T$ such that 
 \[
 \supp_{(A,w)}(P')\subseteq \bigcup_{i=1}^n\supp_{(A,w)}(P(i))=\supp_{(A,w)}(P) \quad\text{where }P=\bigoplus_{i=1}^n P(i)\,.
 \]
 Now it remains to apply \cref{l_support_containment} to deduce that 
\[
P'\in \thick_{(A,w)}^{\odot}(P)\subseteq \T
\] 
where the subset containment uses that $\T$ is a tensor submodule of $\perf(A,w)$.
\end{proof}

The next result specializes to the results of Hopkins~\cite[Theorem~11]{Hopkins} and Neeman~\cite[Theorem~1.5]{Neeman92}; this is also analogous to the classification of Thomason~\cite[Theorem~3.15]{Thomason} in the geometric setting, where one needs that the thick subcategories are closed under a tensor action as well. See also the theorem of Hirano~\cite[Theorem~1.1]{Hirano}, also in a geometric context for matrix factorizations. 

\begin{corollary}\label{cor_tt}
     In the notation and setting of \cref{classification}, there is a bijection of lattices: \[
\begin{tikzcd}
\biggm\{\begin{matrix}\text{Thick tensor ideals}\\ \text{of } \perf(A,0) \end{matrix} \biggm\}\arrow[r,shift left=0.8ex,"\sigma"]&\biggm\{\begin{matrix}\text{Specialization closed subsets}\\ \text{of } \spec(A)\end{matrix} \biggm\}\arrow[l,shift left=0.8ex,"\tau"]
\end{tikzcd}\,.
\]
\end{corollary}
\begin{proof}
    This is obtained by applying the theorem with $w=0$, and noting that a tensor submodule of $\perf(A,0)$ is exactly a thick tensor ideal. Furthermore, when $w=0$ it follows from \cref{sing_locus} that $\supp(A,0)=\spec(A)$. 
    \end{proof}

\begin{corollary}\label{classification_reg}
    If, in addition to the hypotheses of \cref{classification}, $A$ is regular, then there is a bijection of lattices: \[
\begin{tikzcd}
\biggm\{\begin{matrix}\text{Thick subcategories}\\ \text{of } \perf(A,w) \end{matrix} \biggm\}\arrow[r,shift left=0.8ex,"\sigma"]&\biggm\{\begin{matrix}\text{Specialization closed subsets}\\ \text{of } \supp(A,w)\end{matrix} \biggm\}\arrow[l,shift left=0.8ex,"\tau"]
\end{tikzcd}\,.
\]
\end{corollary}

\begin{proof}
    By \cref{p:perf}, the tensor unit $A$ in $\perf(A,0)$ is a classical generator. Hence, every thick subcategory of $\perf(A,w)$ is a tensor submodule; cf.\@ \cref{tensor_module}. Therefore the desired result follows immediately from \cref{classification}.
\end{proof}

\begin{remark}
\label{remark_ben}
The previous corollary is a recent result of Briggs in \cite[Theorem~A.11]{Briggs-Walker}. It is worth mentioning that the result in loc.\@ cit.\@ is proved using a theorem where instead one asks that certain internal Hom's are perfect dg $A$-modules (which is satisfied when $A$ is regular) to conclude that support of endomorphism dg modules detects thick subcategories in $\perf(A,w)$. More precisely, if $P,P'$ are in $\perf(A,w)$ with $\Hom_A(P,P)$, $\Hom_A(P',P)$ and $\Hom_A(P',P')$  perfect dg $A$-modules then 
\[
\supp_A \Hom_A(P,P)\subseteq \supp_A\Hom_A(P',P')\iff P\in \thick_{(A,w)}(P')\,;
\]
see \cite[Theorem~A.4]{Briggs-Walker}.
\end{remark}

The last corollary of the section determines $\perf(A,w)$ admits a $\odot$-generator exactly when $\supp(A,w)$ is a closed subset. This applies, for instance, when $w=0$ or when $w$ is a regular element and $A$ is a smooth over a graded field. 
\begin{corollary}\label{generators}
       In the notation and setting of \cref{classification},  $\supp(A,w)$ is closed if and only if $\perf(A,w)$ admits $\odot$-generator, i.e., there exists $P$ in $\perf(A,w)$ such that $\thick_{(A,w)}^\odot(P)=\perf(A,w)$.

            If, furthermore, $A$ is regular, then these conditions are also equivalent to $\perf(A,w)$ admitting a classical generator. 
\end{corollary}
\begin{proof}
    The first equivalence is an immediate consequence of the classification in \cref{classification}, as $\tau$ takes closed subsets to $\thick_{(A,w)}^\odot(P)$ for some $P$ in $\perf(A,w)$. When $A$ is regular, there is no difference between thick subcategories and thick tensor submodules; see \cref{p:perf,tensor_module}.
\end{proof}
%%%%%%%%%%%%%%%%%%%%%%%%%%%%%%%%%%%%%%%%%%%%%%%%%%%%%%
\section{Thick subcategories over a Koszul complex}\label{Section-thick subcategories}
%%%%%%%%%%%%%%%%%%%%%%%%%%%%%%%%%%%%%%%%%%%%%%%%%%%%%%
In this section we prove \cref{main_intro_thm} from the introduction; cf.\@ \cref{main_result}. First, we introduce the notation used throughout the rest of the section. 

Fix a commutative noetherian ring $R$. We will assume $R$ is regular, in the usual sense that $R_\p$ is a regular local ring for all $\p\in \spec R,$ or equivalently, $\Df(R)=\perf(R).$
We also fix a finite list of elements $\f=f_1,\ldots,f_n$ in $R$, and let $E$ denote the Koszul complex on $\f$ over $R$. This will be regarded as a dg $R$-algebra in the usual way. That is to say, as an $R$-algebra 
\[
E\cong \bigwedge_R (Re_1\oplus \cdots \oplus Re_n)
\]
and the differential is determined $\del e_i=f_i$ and the Leibniz rule. We also let $A=R[\chi_1,\ldots,\chi_n]$ where $|\chi_i|=-2$, and set 
\[
w=f_1\chi_1+\cdots +f_n\chi_n\in A_{-2}\,.
\]

The proof of \cref{main_result} makes use of \cref{classification_reg} and a curved BGG correspondence of Martin~\cite{Martin}; this extends a dg version of the BGG correspondence from \cite[Theorem~7.4]{ABIM}, that holds over a base field, to regular base rings.

\begin{chunk}\label{curved_dg}
    First, define $X=\Hom_R(A,R)\otimes_R E$ equipped with the degree $-1$ endomorphism 
    \[
    \del^X=\id^{\Hom_R(A,R)}\otimes \del^E-\sum_{i=1}^n\chi_i\otimes e_i\,,
    \]
    where the second term is interpreted as left multiplication on $X$. Note that $X$ is a curved dg $A$-module with curvature $-w.$ Given a dg $E$-module $M$, then equip $\Hom_E(X,M)$ with the degree $-1$ endomorphism 
    \[
    \del^{\Hom(X,M)}(\alpha)=\del^M\alpha-(-1)^{|\alpha|}\alpha\del^X\quad\text{ for each }\alpha\in \Hom_E(X,M)\,.
    \]
    This assignment defines an exact functor $\mathsf{h}\colon \Df(E)\to \perf(A,w)$; cf.\@ \cite[Lemma~5.4]{Martin}. Furthermore, in loc.\@ cit.\@ it is shown that $\mathsf{h}M$ can be identified, up to  isomorphism, with the following object in $\perf(A,w)$:
    \[
    A\otimes_R F\quad\text{with}\quad \del=\id^A\otimes \del^F+\sum_{i=1}^n\chi_i\otimes e_i\,,
    \]
    where $F\simeq M$ such that $F$ is a dg $E$-module that is a bounded complex of finite free $R$-modules. The main result of Martin's thesis is that $\mathsf{h}\colon \Df(E)\to \perf(A,w)$ is an equivalence of triangulated categories. 
\end{chunk}

Finally, recall the definition of $\supp(A,w)$ from \cref{defn_sing}.
    A set-theoretic description of $\supp(A,w)$ is given in \cref{decomposition}. 

\begin{theorem}\label{main_result}
In the above notation,  there are inclusion preserving bijections 
\[
\begin{tikzcd}
\biggm\{\begin{matrix}\text{Thick subcategories}\\ \text{of } \Df(E) \end{matrix} \biggm\}\arrow[r,shift left=0.8ex,"\sigma"]&\biggm\{\begin{matrix}\text{Specialization closed}\\ \text{subsets of } \supp(A,w)\end{matrix} \biggm\}\arrow[l,shift left=0.8ex,"\tau"]
\end{tikzcd}\,,
\]
where the inverse bijections are given by
\[
\sigma(\T)=\bigcup_{M\in \T}\supp_{(A,w)}(\mathsf{h}M) \quad\text{and}\quad \tau(\mathcal{V})=\{M\in \Df(E)\mid \supp_{(A,w)}(\mathsf{h}M)\subseteq \mathcal{V}\}\,.
\]
\end{theorem}

\begin{proof}[Proof of \cref{main_result}]
   The curved BGG equivalence $\mathsf{h}\colon \Df(E)\to \perf(A,w)$ in \cref{curved_dg} establishes a bijection \[\biggm\{\begin{matrix}\text{Thick subcategories}\\ \text{of } \Df(E) \end{matrix} \biggm\}\xlongrightarrow{\tilde{\mathsf{h}}} \biggm\{\begin{matrix}\text{Thick subcategories}\\ \text{of } \perf(A,w)\end{matrix} \biggm\}\,.\]
    It remains to apply \cref{classification_reg}, and notice the desired bijection is witnessed by the maps $\sigma$ and $\tau$  as claimed. 
    \end{proof}

    By the curved BGG correspondence in \cref{curved_dg} and \cref{e_tt}, one can upgrade the triangulated structure on $\Df(E)$ to a tensor triangular one when $\f=0.$ Specializing to this case, we obtain the following corollary. 

\begin{corollary}\label{cor_zero}
     In the notation and setting of \cref{main_result}, if $\f=0$, then there is a bijection of lattices: \[
\begin{tikzcd}
\biggm\{\begin{matrix}\text{Thick subcategories}\\ \text{of } \Df(E) \end{matrix} \biggm\}\arrow[r,shift left=0.8ex,"\sigma"]&\biggm\{\begin{matrix}\text{Specialization closed subsets}\\ \text{of } \spec(A)\end{matrix} \biggm\}\arrow[l,shift left=0.8ex,"\tau"]
\end{tikzcd}\,.
\]
\end{corollary}
\begin{proof}
    This is an immediate consequence of \cref{cor_tt,main_result}.
    \end{proof}
The next corollary was first established by Stevenson in \cite[Theorem~4.9]{Stevenson:2014b}. It was also recently obtained by Briggs in \cite[Example~A.14]{Briggs-Walker}.
\begin{corollary}\label{cor_ci}
  In the notation and setting of \cref{main_result}, if   $\f$ is an $R$-regular sequence, then there is a bijection of lattices: \[
\begin{tikzcd}
\biggm\{\begin{matrix}\text{Thick subcategories}\\ \text{of } \Df(R/(\f)) \end{matrix} \biggm\}\arrow[r,shift left=0.8ex,"\sigma"]&\biggm\{\begin{matrix}\text{Specialization closed subsets}\\ \text{of } \supp(A,w)\end{matrix} \biggm\}\arrow[l,shift left=0.8ex,"\tau"]
\end{tikzcd}\,.
\]
\end{corollary} 
\begin{proof}
    As $\f$ is an $R$-regular sequence, the augmentation map $E\to R/(\f)$ is a quasi-isomorphism of dg algebras. This induces an equivalence of triangulated categories $\Df(E)\equiv \Df(R/(\f))$, and so now the desired classification is obtained from \cref{main_result}.
\end{proof}

\begin{corollary}
\label{cor_generator}
In the notation and setting of \cref{main_result}, $\Df(E)$ has a classical generator if and only if $\supp(A,w)$ is a closed subset. In particular, if (i) $\f=0$ or (ii) $\f\neq 0$ and $R$ is quasi-excellent and local, then $\Df(E)$ has a classical generator. 
\end{corollary}

\begin{proof}
    This is a direct consequence of \cref{generators,main_result}. 
\end{proof}

Next we clarify the connection between the theory of support on $\Df(E)$ obtained using the curved BGG correspondence and the theory of cohomological support defined on $\Df(E)$ from \cite{Pol,Po}; see also \cite{Av2,ABInvent, Jorgensen:2002}.
\begin{chunk}\label{Hochschild}
Let $\env[R]{E}\coloneqq E\otimes_R E$ be the enveloping dg algebra of $E$ over $R$, and regard $E$ as a dg $\env[R]{E}$-module via the multiplication map $\env[R]{E}\rightarrow E$.
Recall that the Hochschild cohomology of $E$ over $R$ is 
$
\hc(E/R)\coloneqq \Ext_{\env[R]{E}}(E,E).
$
By \cite[2.9]{ABJA}, the graded algebra $\hc(E/R)$ is isomorphic to $\h(E)[\chi_1,\ldots,\chi_n]$, where $|\chi_i|=-2$. 

For each $M$ in $\D(E)$, the left derived functor $-\lotimes_EM\colon \D(\env[R]{E})\rightarrow \D(E)$ induces a homomorphism of graded $R$-algebras
$$
A\rightarrow  \hc(E/R)\xrightarrow{-\lotimes_EM}\Ext_E(M,M)\,.
$$
 This defines a central action of $A$ on $\D(E)$; 
see \cite[Section 3]{AS}. The  \textbf{cohomological support} of a pair $(M,N)$ over $E$ is
\[
\V_E(M,N)\coloneqq\supp_A \Ext_E(M, N)\,,
\]
and set $\V_E(M)\colonequals\V_E(M,M)$. 

When $M,N$ are in $\Df(E)$, by \cite[Theorem 4.3.2]{Po}, $\Ext_{E}({M},{N})$ is finitely generated over $A$, and hence  $\V_E(M,N)$ is a closed subset of $\spec A$. 
\end{chunk}

\begin{proposition}
    \label{coincide}
    For each $M,N$ in  $\Df(E)$, one has \[\V_E(M,N)=\supp_{(A,0)}(\Hom_A(\mathsf{h}M,\mathsf{h}N))\,.\]
    As a consequence, $\V_E(M)=\supp_{(A,w)}(\mathsf{h}M)$
\end{proposition}
\begin{proof}
    The curved BGG equivalence $\mathsf{h}$ from \ref{curved_dg} induces an $A$-linear isomorphism  
    \[
    \Ext_E(M,N)\cong \h(\Hom_A(\mathsf{h}M,\mathsf{h}N))\,,
    \]
    and so the first equality follows immediately. The second equality follows from the first using that 
    \[
    \supp_{(A,0)}(\Hom_A(P,P))=\supp_{(A,w)}P
    \]
    which is a consequence of equality holding in \cref{support}(4) for each $P$ in $\perf(A,w)$.
\end{proof}

\begin{proposition}\label{support_vars_thick}
    For $M,N$ in $\Df(E)$, one has $M$ is in $\thick_{\D(E)}(N)$ if and only if $\V_E(M)\subseteq \V_E(N)$. 
\end{proposition}
\begin{proof}
    Using the classification in \cref{main_result}, one only needs to refer to \cref{coincide}. 
\end{proof}

The next result specializes to \cite[Corollary~4.11]{Stevenson:2014b} when $\f$ is a regular sequence; see also \cite[Theorem~1]{LP}. 

\begin{corollary}\label{duality}
 In the notation and setting of \cref{main_result}, if $M$ is in $\Df(E)$, then 
 \[
 \thick_{\D(E)}(M)=\thick_{\D(E)}(\RHom_E(M,E))\,.
 \]
Furthermore, the involution $\RHom_E(-,E)$ on $\Df(E)$ fixes the lattice of thick subcategories of $\Df(E)$.
\end{corollary}
\begin{proof}
   Let $(-)^\vee=\RHom_E(-,E)$, and observe that 
   \[
   \Ext_E(M,M)\cong \Ext_E(M^\vee,M^\vee)\,.
   \]
   Thus, $\V_E(M)=\V_E(M^\vee)$ and so the desired result follows from \cref{coincide,support_vars_thick}.
\end{proof}

Recall the complexity of $M,N$ in $\Df(E)$ is 
\[
\cx_E(M,N)= \inf\{c:\exists a>0\text{ such that }an^{c-1}\geqslant \rank_k (\Ext_E^n(M,N)\otimes k)\text{ for }n\gg 0\}
\]
The next result was shown in \cite[Theorem 5.3.1]{Po} using a different approach; this also recovers a result of Avramov--Buchweitz in \cite[Corollary~5.7]{ABInvent} (see also \cite[Theorem 2]{LP}, which argues in a similar fashion to the way below).

\begin{corollary}\label{cor_sym_cx}
  For $M,N$ in $\Df(E)$, one has $\cx_E(M,N)=\cx_E(N,M)$. 
\end{corollary}
\begin{proof}
Let $(-)^\vee=\RHom_E(-,E)$. 
 If one considers the category $\T$ of all $X$ in $\Df(E)$ such that $\cx_E(M,X)\leqslant \cx_E(M,N)$, then $\T$ is thick. Also, as $N$ is in $\T$ we can use \cref{duality} to obtain that $N^\vee$ belongs to $\T$. Thus, $\cx_E(M,N^\vee)\leqslant\cx_E(M,N)$. Repeating this argument (three times) it follows that 
 \[
 \cx_E(M,N)=\cx_E(M^\vee,N^\vee)\,.
 \]
 Now it remains to notice that Grothendieck duality establishes the isomorphism $\Ext_E(M^\vee,N^\vee)\cong \Ext_E(N,M)$ and hence, $\cx_E(M,N)=\cx_E(N,M)$. 
\end{proof}

\begin{remark}
Moving beyond when $R$ is regular, one could consider $\Df(E/R)$ as in \cite{BMP,BW,Iyengar-Pollitz-Sanders}, where many results still hold in the \emph{relative} setting.  The objects of $\Df(E/R)$ are those of  $\Df(E)$ that are perfect over $R$. In this setting, the curved BGG functor from \cref{curved_dg} restricts to an exact embedding  $\Df(E/R)\to \perf(A,w)$. However, it is unclear whether the essential image is a tensor submodule of $\perf(A,w)$. If it were, then 
it seems likely that the arguments in this article would apply to classify the thick tensor submodules of $\Df(E/R)$, and show things like 
\[
\V_E(M)\subseteq \V_E(N)\iff M\in\thick_{\D(E)}^{\odot}(N)\,.
\]
\end{remark}
We end the article by giving a bijective correspondence between \(\Sing(A, w)\) and the disjoint union of the specific homogeneous spectrum of graded polynomial algebras, as well as some concluding remarks.

For each $\p\in \spec R$, set \[c(\p)\colonequals n-(\dim R_\p-\embdim(R_\p/\f R_\p))\,,\] where $\dim$ and $\embdim$ represents the Krull dimension and the embedding dimension respectively. If $\f$ is an $R$-regular sequence, then $c(\p)$ coincides with codimension of the complete intersection ring $(R/(\f))_\p$.

\begin{proposition}\label{decomposition}
There is a bijection
$$
   \supp(A,w)\cong\bigsqcup_{\p\in \spec R \text{ \rm and } (\f)\subseteq \p}\spec S(\p),
$$
where  $S(\p)=\rm{Sym}_{\kappa(\p)}(\shift^{-2}\kappa(p)^{c(\p)})$. 
\end{proposition}
\begin{proof}
Observe that 
\begin{align*}
    \supp(A,w)&=\bigcup_{M\in \Df(E)}\V_E(M)\\
    &=\bigsqcup_{\p\in \spec {R}}\bigcup_{M\in \Df(E)}\V_{{E}_\p}({M}_\p,\kappa(\p))\,,
\end{align*}
where the first equality holds by \cref{main_result,coincide}, and the second equality is from  \cite[Theorem 11.3]{BIKnorm}. Also,  note that when $\f\not\subseteq \p$, then $E_\p$ is contractible and so each $\V_{{E}_\p}({M}_\p,\kappa(\p))$ is empty.  Hence, we can assume $\f\subseteq \p$ and we are done upon noting that 
\[
\V_{{E}_\p}({M}_\p,\kappa(\p))\subseteq \spec S(\p)=\V_{E_\p}((R/\p)_\p, \kappa(\p)))\,;
\]
 ones makes use of a functionality of cohomological supports in \cite[Proposition~3.4.3]{Po} for the containment, and  \cite[Example~4.1.4(2)]{Po} to obtain the equality.
\end{proof}
\begin{remark}
    Let $\Proj A/(w)$ denote the set of homogeneous prime ideals that do not contain $A/(w)_{<0}$. Define $\Sing(\Proj A/(w))\colonequals \supp(A,w)\cap \Proj A/(w)$. It is straightforward to verify that  \[\supp(A,w)=\Sing(\Proj A/(w))\bigsqcup \{(\p,\chi_1,\cdots,\chi_n)\mid \p\in \spec R \text{ and }(\f)\subseteq \p\}\,.\]
    Combining this with Proposition \ref{decomposition}, there is a bijection
$$
\Sing (\Proj A/(w))\cong \bigsqcup_{\p\in \spec R\text{ and } (\f)\subseteq \p}\Proj S(\p)\,;
$$
this was proved by Stevenson \cite[Proposition 10.4]{Stevenson:2014a} when $\f$ is a regular sequence. 
\end{remark}

\begin{remark}   \label{remark_ben_jian_josh}
In this article, we have focused on the classification of thick subcategories of $\Df(E)$. In the spirit of \cite{CI}, we avoid infinite constructions to achieve this result; the former classified thick subcategory of $\Df(A)$ when $A$ is an artinian complete intersection ring without the use of any infinitely generated modules (over $\Ext_A(k,k)$). However, in light of the approach of Neeman~\cite{Neeman92}, later clarified by many others like  \cite{BIKann, BIKtop, Stevenson:2013a, Stevenson:2014a} and more recently in \cite{Lattices,Bartheletal,BIKPfiber}, it is natural to wonder if $\Df(E)$ can be regarded as the compact objects in a certain ``big" triangulated category $\mathsf{K}$ and deduce the classification of thick subcategories in \cref{main_result} from a classification of the localizing subcategories of $\mathsf{K}$. In upcoming work joint with Briggs, we flesh out this approach by suitably extending Martin's thesis to an equivalence, similar to how Stevenson extended an equivalence of Orlov~\cite[Theorem 2.1]{Orlov06} to certain \emph{big} categories; cf.\@ \cite[Proposition 8.7]{Stevenson:2014a}.
\end{remark}

\section*{Acknowledgments}
We thank Benjamin Briggs for several helpful comments on the manuscript; furthermore, we are very grateful to Briggs for conversations that improved the article and clarified the  connections between some of the results in this work and the appendix in \cite{Briggs-Walker}.
The second author thanks James Cameron, Sarasij Maitra, and Tim Tribone for a number of enlightening discussions, specifically on how tensor-nilpotence holds over curved algebras in \cref{l_tensor_nilpotence}. 

The authors would also like to thank an anonymous referee for carefully reading the manuscript and providing valuable comments.

The first author was supported by the National Natural Science Foundation of China (No.\@ 12401046) and the Fundamental Research Funds for the Central Universities (No.\@ CCNU24JC001, CCNU25JC025, CCNU25JCPT031). The second author was supported by NSF grant DMS-2302567.

\bibliographystyle{amsplain}
\bibliography{ref}
\end{document}